\documentclass[11pt]{article}
\usepackage{amsmath, amssymb, amsthm}
\usepackage{enumerate}
\usepackage{amsmath, amssymb, amsthm}
\usepackage{enumerate}

\newtheorem*{propositiona}{Proposition}
\newtheorem*{theorema}{Lemma}
\newtheorem*{theoremb}{Proposition}
\newtheorem*{theoremc}{Theorem 3}
\newtheorem*{theoremd}{Change of Variable Formula}
\newtheorem*{theoreme}{Substitution Formula}
\newtheorem*{theoremf}{Proof, Substitution Formula}
\newcommand{\R}{{\mathbb{R}}}
\newcommand{\taf}{{\hskip 5pt} $\blacksquare$
                  \renewcommand{\qedsymbol}{}}
\begin{document}
\title{The change of variable formula for the Riemann-Stieltjes integral}
\author{Alberto Torchinsky}

\date{April 1, 2019}



\maketitle
\begin{abstract}
We  consider   general  formulations of the  
change of variable formula for the Riemann-Stieltjes integral, including the case when the substitution is not invertible.
\end{abstract}
\section{Introduction.}
This note concerns  general  formulations of the  
change of variable, or substitution, formula for the Riemann-Stieltjes integral. A prototype of our results is the following,
\begin{theoreme}  
 Let $\varphi$ be a bounded,  Riemann integrable function defined on an interval $I=[a,b]$ that does not change sign on $I$, and let $\Phi$ be an indefinite integral of $\varphi$ on $I$.
Let $\psi$ be a bounded, Riemann integrable function defined on $\Phi(I)$, the range of $\Phi$, 
and let $\Psi$ be an indefinite integral of $\psi$ on $\Phi(I)$.

Then, if a bounded function $f$  defined on $\Phi(I)$ is   Riemann  integrable with respect to $\Psi$ on $\Phi(I)$, $ f(\Phi)\psi(\Phi) $ is Riemann  integrable with respect to $\Phi$ on $I$, and in that case, with $\mathcal I=[\Phi(a),\Phi(b)],$ 
\begin{equation} \int_{\mathcal I} f d\Psi= \int_I  f(\Phi)\,\psi(\Phi) d\Phi.
\end{equation}
\end{theoreme}

Thus, the substitution formula holds when the Riemann-Stieltjes integral is computed with respect to an arbitrary function $\Psi$, and the  substitution $\Phi$ is invertible. Together with the change of variable formula established below, which holds when $\Psi$ is monotone, or a difference of monotone functions,  and $\Phi$ is not necessarily invertible, these formulas constitute the main results in this note. Though it seems that these  formulas should be  known, they are not present in the classic or standard literature in the area.

We will begin by introducing the necessary definitions and notations. 
Fix a closed finite interval  $I=[a,b]\subset \R$, and let $\Phi$ be a continuous monotone (increasing) function defined on $I$. For a partition $\mathcal P=\{I_k\}$ of $I$, where $I_k=[x_{k,l},x_{k,r}]$, and a bounded function $f$  on $I$, let  
 $U(f,\Phi, \mathcal{P})$ and  $L(f,\Phi, \mathcal{P})$ denote  the  upper and lower Riemann sums of $f$ with respect to $\Phi$ on $I$ along $\mathcal{P}$,  i.e.,
\[U(f,\Phi, \mathcal{P})=\sum_k \big(\sup_{I_k}f\big)\, \big(\Phi(x_{k,r})- \Phi(x_{k,l})\big),\]
and  
\[ L(f,\Phi, \mathcal{P})=\sum_k \big(\inf_{I_k}f\big)\, \big(\Phi(x_{k,r})- \Phi(x_{k,l})\big),
\]
respectively, and set
 \[ U(f,\Phi)=\inf_{\mathcal{P}}\, U(f,\Phi,\mathcal{P}),\quad{\rm{and}} \quad
 L(f,\Phi)=\sup_{\mathcal{P}} L(f,\Phi, \mathcal{P}).
 \]
 
We say that $f$ is Riemann integrable with respect to $\Phi$ on $I$ if $U(f,\Phi)=L(f,\Phi)$, and in this case the common value is denoted $\int_I f\,d\Phi$, the Riemann integral of $f$ with respect to $\Phi$ on $I$. 

When $\Phi(x)=x$ one gets the usual Riemann integral on $I$, and  $\Phi$ is omitted in the above notations. And, throughout this note, when it is clear from the context, integrable means Riemann integrable with respect to $\Phi(x)=x$, and Riemann-Stieltjes integrable means integrable with respect to a  general $\Phi$.
 
The following are working characterizations of integrability \cite {Bruckner}, \cite{Hunter},  \cite{Torch}. 
A bounded function  $f$ defined on $I$ is Riemann integrable with respect to $\Phi$  on $I$ iff,  given $\varepsilon >0$, there is a partition ${\mathcal{P}}$ of $I$, which 
may  depend on $\varepsilon$,  such that 
\begin{equation}  U(f,\Phi,{\mathcal{P}})-L(f,\Phi,{\mathcal {P}})\le \varepsilon.
\end{equation}

Furthermore, a sequential characterization holds, to wit,
 (2) is equivalent to the existence of a sequence  of partitions $\{\mathcal P_n\}$ of $I$ such that
\begin{equation*}
\lim_{n}  \big( U(f,\Phi,\mathcal {P}_n) - L(f,\Phi, \mathcal {P}_n) \big) = 0,
\end{equation*}
and  in this case  
\begin{equation} \lim_{n} U(f,\Phi, \mathcal P_n) = \lim_{n}L(f, \Phi, {\mathcal P}_n)=
\int_I f\,d\Phi.
\end{equation}

Also, integrability can be characterized in terms of  oscillations. Recall that, given a bounded function $f$ defined on $I$  and an interval $J\subset I$, the oscillation ${\rm {osc\, }}(f, J)$ of $f$ on $J$ is defined as 
$ {\rm {osc\, }}(f, J) = \sup_J f - \inf_J f$. Then, a bounded function $f$ is Riemann integrable with respect to $\Phi$ on $I$ iff,  given $\varepsilon >0$, there is a partition ${\mathcal{P}}=\{I_k\}$ of $I$, which 
may  depend on $\varepsilon$,  such that 
\begin{equation}  \sum_k {\rm {osc\, }}( f, I_k)\,\big(\Phi(x_{k,r})- \Phi(x_{k,l})\big)\le\varepsilon.
\end{equation}

And, a sequential characterization  holds, to wit,  (4) is equivalent to the existence of a sequence of partitions  $\{\mathcal {P}_n\}$  of $I$ consisting of the intervals $\mathcal {P}_n=\{I_k^n\}$ with $  I_k^n = [x_{k,l}^n, x_{k,r}^n]$, such that 
\begin{equation} \lim_n \sum_k {\rm {osc\, }}( f, I_k^n)\,\big(\Phi(x_{k,r}^n)- \Phi(x_{k,l}^n)\big)=0.
\end{equation}

These characterizations do not necessarily hold if $\Phi$ fails to be monotone. 
Moreover, note that if (2) holds for a  partition ${\mathcal P}$, 
it also holds for  partitions  ${\mathcal{P}'}$ 
 finer than ${\mathcal{P}}$. Invoking (36), this  observation  applies to other concepts as well, including (3), (4), and (5).

Finally, since $\Phi$  is continuous and increasing on $I$, $\Phi(I)$
is an interval $\mathcal I=[\Phi(a), \Phi(b)]$  with endpoints $\Phi(a)$ and $ \Phi(b)$.  Note that each interval $\mathcal J=[y_1,y_2] \subset \mathcal I$ is of the form $[\Phi(x_1),\Phi(x_2)]$, where $\Phi(x_1)=y_1, \Phi(x_2)=y_2$, and $[x_1, x_2]$ is a subinterval of $I$. Moreover, partitions $\mathcal P$ of $I$  induce a corresponding partition $\mathcal Q$ of $\mathcal I$, and, conversely,  every partition of $\mathcal I$ can be expressed as $\mathcal Q$ for some partition $\mathcal P$ of $I$.

\section{The Substitution Formula.}

We begin by proving a  result that includes the familiar  substitution formula \cite{Apostol},
\begin{theoremb} Let $\Phi$ be a continuous  monotone function defined on $I$, and $\Psi$ 
defined on $\mathcal I= \Phi(I)$. Let $f$ be a bounded function on $\mathcal I$.\!\!  Then,  $f$ is Riemann integrable with respect to $\Psi$ on $\mathcal I$ iff $ f(\Phi)$ is Riemann  integrable with respect to $\Psi(\Phi)$ on $I$,
and  in that case  we have 
\begin{equation} \int_{\mathcal I} f\, d\Psi=\int_I f(\Phi)\, d\Psi(\Phi).
\end{equation}
\end{theoremb}

\begin{proof}
Specifically,  (6) means that, if the integral on either side of the equality exists, so does the integral on the other side and they are equal. 
To see this, let the  partition  $\mathcal Q =\{\mathcal I_k\}$ of $\mathcal I$ correspond to the partition $\mathcal P=\{ I_k\}$ of $I$ such that $\mathcal I_k =[\Phi(x_{k,l}),\Phi(x_{k,r}) ]$, where $I_k=[x_{k,l}, x_{k,r}]$. Then, since
$  \sup_{{\mathcal I}_k} f = \sup_{I_k} f(\Phi)$, 
it readily follows that
\begin{align*}
U(f,\Psi, &\mathcal Q) = \sum_k \big(\sup_{{\mathcal I}_k} f\big)\, \big(\Psi( \Phi(x_{k,r}))-
\Psi(\Phi(x_{k,l}))\big)\\
&= \sum_k 
 \big(\sup_{I_k} f(\Phi)\big)\, \big(\Psi(\Phi(x_{k,r})) -\Psi(\Phi(x_{k,l}))\big) = U(f(\Phi), \Psi(\Phi), \mathcal P),
\end{align*}  
and, similarly,
$ L(f,\Psi, \mathcal Q)= L(f(\Phi), \Psi(\Phi), \mathcal P).$ (6) follows at once from these identities.
\taf
\end{proof}

Next we consider the   particular case of the substitution formula  when both $\varphi$ and $\psi$ are of constant sign. In this instance we have,

\begin{theorema} 
Let $\varphi$ be a bounded,  Riemann integrable function defined on an interval $I=[a,b]$ that does not change sign on $I$, 
and let $\Phi$ be an indefinite integral of $\varphi$ on $I$.
Let $\psi$ be a bounded,  Riemann integrable function defined on $\Phi(I)$, the range of $\Phi$, that does not change sign, 
and let $\Psi$ be an indefinite integral of $\psi$ on $\Phi(I)$.

Let  $f$ be bounded on $\Phi(I)$. Then,  $f$ is   integrable with respect to $\Psi$ on $\Phi(I)$   iff $ f(\Phi)\psi(\Phi)$  is  integrable with respect to $\Phi$ on $I$, and  in that case, with $\mathcal I=[\Phi(a), \Phi(b)]$, 
\begin{equation} \int_{\mathcal I} f\,d\Psi= \int_I  f(\Phi)\,\psi(\Phi)\,d\Phi.
\end{equation}
\end{theorema}

\begin{proof} It suffices to prove the  result when $\varphi, \psi$ are positive. Indeed, if the result holds in this case, when  $\varphi$ is negative it follows by replacing $\varphi$ by $-\varphi$,  $\psi(x)$ by $\psi(-x)$, and $f(y)$ by $f(-y)$  in (7), and  when $\psi$ is negative, by replacing $\psi$ by $-\psi$ in (7). 

Now, by assumption  we have, 
\begin{equation*} \Phi(x)=\Phi(a)+ \int_{[a,x]} \varphi,\quad  x\in I, 
\end{equation*}
and, 
\begin{equation*} \Psi(y)=\Psi(\Phi(a))+ \int_{[\Phi(a),y]} \psi,\quad  y\in \Phi(I).
\end{equation*}

First, assume that $ f(\Phi) \psi(\Phi) $ is Riemann-Stieltjes integrable, and fix $\varepsilon >0$. Then, for a partition $\mathcal P=\{I_k\}$ of  $I$ with $I_k=[x_{k,l},x_{k,r}]$ and 
$\mathcal I_k=[ \Phi(x_{k,r}), \Phi(x_{k,l})]$, pick $ \xi_k\in I_k$ such that
\begin{equation} U( f(\Phi), \Psi(\Phi), \mathcal P)\le \sum_k  f(\Phi(\xi_k))\,\big(  \Psi(\Phi (x_{k,r})) - \Psi(\Phi(x_{k,l}))\big)+\varepsilon.
\end{equation}

Now,   the sum on the right-hand side of (8)  equals
\begin{align}\sum_k   f(\Phi (\xi_k)&)\int_{\mathcal I_k} \big(\psi - \psi(\Phi(\xi_k)) \big)\nonumber\\
 &+\sum_k   f(\Phi(\xi_k))\,  \psi(\Phi(\xi_k))  \big( \Phi (x_{k,r})- \Phi (x_{k,l})\big)
=A+B,
\end{align}
say, where clearly, with $M_f$ a bound for $f$,  $ A\le M_f \sum_k {\rm osc}\,(\psi,\mathcal I_k)\,|\mathcal I_k|$, and 
$ B \le U\big( f(\Phi)  \psi(\Phi), \Phi,\mathcal P\big)$.
Thus, by (8) and (9),
\begin{equation} U( f(\Phi), \Psi(\Phi), \mathcal P)\le M_f \sum_k {\rm osc}\,(\psi,\mathcal I_k)\,|\mathcal I_k|
+ U( f(\Phi)  \psi(\Phi), \Phi,\mathcal P)+\varepsilon.
\end{equation}

Applying (10) to $-f$ gives  
\begin{equation*} -L( f(\Phi), \Psi(\Phi), \mathcal P) \le 
 M_f \sum_k  {\rm osc}\,(\psi,\mathcal I_k)\,|\mathcal I_k|  - L( f(\Phi)  \psi(\Phi), \Phi,\mathcal P) +\varepsilon,
\end{equation*}
which added to (10)  yields
\begin{align} U( f(\Phi),\Psi(\Phi), \mathcal P) &- L( f(\Phi), \Psi(\Phi), \mathcal P)
\le 2 M_f \sum_k {\rm osc}\,(\psi,\mathcal I_k)\,|\mathcal I_k|\nonumber\\ 
&+\big(
U( f(\Phi)  \psi(\Phi), \Phi,\mathcal P) -  L( f(\Phi)  \psi(\Phi), \Phi,\mathcal P)\big)  +2\,\varepsilon.
\end{align}

Let the  partition $\mathcal {P}$  of $I$  satisfy simultaneously 
(4) for $\psi$ and (2) for $ f(\Phi)  \psi(\Phi)$ with respect to $\Phi$ for the $\varepsilon>0$ picked at the beginning of the proof; a common refinement of a partition that satisfies  
(4) for $\psi$ and one that satisfies (2) for $ f(\Phi)  \psi(\Phi)$ with respect to $\Phi$ will do.   Then from (11) it  follows that $U( f(\Phi),  \Psi(\Phi), \mathcal P) - L( f(\Phi), \Psi(\Phi),  \mathcal P) \le 2 M_f\,\varepsilon + \varepsilon +2\, \varepsilon,
$ and,  since $\varepsilon>0$ is arbitrary, by (2),  $f(\Phi)$ is Riemann-Stieltjes integrable 
and $U(f(\Phi),\Psi(\Phi))=L(f(\Phi),\Psi(\Phi))=\int_I f(\Phi)\,d\Psi(\Phi)$. 

To evaluate the integral in question, let the sequence of   partitions $\{\mathcal {P}_n\}$
of $I$  satisfy simultaneously (5)  for $\psi$ and (3) for $f(\Phi)\psi(\Phi)$. Then, given $\varepsilon>0$,  from (10) it follows that
\begin{align*}\int_I f(\Phi)\,&d\Psi(\Phi)= U(f(\Phi),\Psi(\Phi))\le \limsup_n U(f(\Phi),\Psi(\Phi), \mathcal P_n)\\
&\le \limsup_n M_f \sum_k  {\rm {osc\, }}(\psi, \mathcal I_k^n)\,|\mathcal I_k^n|+ \limsup_n  U( f(\Phi)\psi(\Phi),\Phi,\mathcal P_n)+\varepsilon,\\
&{\hskip .53in =\int_I f(\Phi)\psi(\Phi)\,d\Phi +\varepsilon,}
\end{align*}
which, since $\varepsilon$ is arbitrary, gives
$ \int_I f(\Phi) d\Psi(\Phi) \le  \int_I f(\Phi)\psi(\Phi)\,d\Phi\,.$ 
Furthermore,   replacing $f$ by $-f$  it follows that $ \int_I f(\Phi)\psi(\Phi)\,d\Phi
\le \int_I f(\Phi)\,d\Psi(\Phi) \,,$ and the integrals are equal. Therefore, since $\Phi$ is continuous, monotone  on $I$, by (6),
\[ \int_{\mathcal I} f d\Psi= \int_I f(\Phi)\,d\Psi(\Phi) =\int_I f(\Phi)\psi(\Phi)\,d\Phi, 
\]
(7) holds, and the conclusion obtains. 

To prove the converse, by (6) and the assertion we just proved, it suffices to show that, if $f(\Phi)$ is integrable   with respect to $\Psi(\Phi)$ on $I$,  $f(\Phi)\psi(\Phi)$ is integrable with respect to $\Phi$ on $I$.  
Let, then,     
 $\mathcal P =\{I_k\}$ be a partition of $I$, and, given $\varepsilon>0$, pick
$\xi_k\in I_k$ such that
\begin{equation}
U (f(\Phi)\psi(\Phi),\Phi,\mathcal P)
\le \sum_k   f(\Phi(\xi_k))  \psi(\Phi(\xi_k)) \big( \Phi (x_{k,r})- \Phi (x_{k,l})\big)+\varepsilon.
\end{equation}

Now, there are two types of summands in (12),  to wit, those where
$f(\Phi(\xi_k))>0$, and those where $ f(\Phi(\xi_k)) <0$. In the former case,    we have 
\begin{align*}f(\Phi(\xi_k)) \big(\psi(\Phi(\xi_k)) &\mp \inf_{\mathcal I_k} \psi\big)\, \big( \Phi (x_{k,r})- \Phi (x_{k,l})\big)
\\ 
&\le f(\Phi(\xi_k))\,{{\rm osc}}\,(\psi,\mathcal I_k)\, |\mathcal I_k|+ f(\Phi(\xi_k))\big( \inf_{\mathcal I_k} \psi\big)\, |\mathcal I_k|,
\end{align*}
where the first term is bounded by $M_f\,{{\rm osc}}\,(\psi,\mathcal I_k)\,|\mathcal I_k|$, and the  second  by
\begin{align*}   f(\Phi(\xi_k))\,\int_{[\Phi (x_{k,l}), \Phi (x_{k,r})]}\psi &=f(\Phi(\xi_k))\,
\big(\Psi(\Phi (x_{k,r}))- \Psi(\Phi (x_{k,l}))\big)\\
&\le \big(\sup_{I_k} f(\Phi)\big)\,\big(\Psi(\Phi (x_{k,r}))- \Psi(\Phi (x_{k,l}))\big)\,.
\end{align*}

Along similar lines,  since    $\int_{\mathcal I_k}\psi\le (\sup_{\mathcal I_k}\psi) \big(
\Phi (x_{k,r}))- \Phi (x_{k,l})\big)$,   in the latter case we have 
\begin{align*}f(\Phi(\xi_k))\psi (\Phi(\xi_k)) &\big( \Phi (x_{k,r})- \Phi (x_{k,l})\big)
\\
=(-f(\Phi(\xi_k)&))(-  \psi(\Phi(\xi_k))  \pm \sup_{\mathcal I_k} \psi\big)\, \big( \Phi (x_{k,r})- \Phi (x_{k,l})\big)
\\
\le M_f\,{{\rm osc}}&\,(\psi,\mathcal I_k)|\mathcal I_k| - f(\Phi(\xi_k))\big(-\sup_{I_k}\psi\big) \big( \Phi (x_{k,r})- \Phi (x_{k,l})\big)  
\\
\le M_f\, &{{\rm osc}}\,(\psi,\mathcal I_k)|\mathcal I_k| + \big(\sup_{I_k}f(\Phi)\big)\, \big( \Psi(\Phi(x_{k,r}))- \Psi(\Phi(x_{k,l}))\big ).
\end{align*}

Whence, combining these estimates it follows that  
\begin{align*}\sum_k   f(\Phi(\xi_k))\, &\psi(\Phi(\xi_k)) \, \big( \Phi (x_{k,r})- \Phi (x_{k,l})\big)
\\
&\le M_f \sum_{k }   {\rm {osc\,}}(\psi, \mathcal I_k ) \,|\mathcal I_k|
+   U( f(\Phi),  \psi(\Phi), \mathcal P)\,,
\end{align*}
and by (12),
\[U (f(\Phi)\psi(\Phi),\Psi(\Phi),\mathcal P)\le M_f \sum_{k }   {\rm {osc\,}}(\psi, \mathcal I_k ) \,|\mathcal I_k|
+   U( f(\Phi),  \psi(\Phi), \mathcal P) +\varepsilon. 
\]
Now, as in (11) it follows that
\begin{align*} U( f(\Phi)  \psi(\Phi), \Phi,\mathcal P) &-  L( f(\Phi)  \psi(\Phi), \Phi,\mathcal P)\big)\le 2\, M_f \sum_k {\rm osc}\,(\psi,\mathcal I_k)\,|\mathcal I_k|\nonumber\\
&+\big(
U( f(\Phi), \Psi(\Phi), \mathcal P) - L( f(\Phi), \Psi(\Phi), \mathcal P)\big)  +2\,\varepsilon,
\end{align*} 
and so, since $\varepsilon>0$ is arbitrary, picking an appropriate $\mathcal P$, by (2) we conclude that $f(\Phi)  \psi(\Phi)$ is Riemann-Stieltjes integrable, and the proof is finished.  
\taf\end{proof}

We are now ready to prove the substitution formula stated in the introduction. In this case $\psi$ is allowed to changes sign on $\Phi(I)$. 
\begin{theoremf}
\end{theoremf}
It suffices to prove the result when  $\varphi$ is positive, for when $\varphi$ is negative  on $I$ the result follows by a direct proof, or simply by replacing $\varphi$ by $-\varphi$, $\psi(x)$ by $\psi(-x)$, and $f(y)$ by $f(-y)$ in (1).
Let $f$ be integrable with respect to $\Psi$ on $\Phi(I)$. The idea is to show that $\int_{\Phi(I)}f \,d\Psi$  can be approximated arbitrarily close by the Riemann sums of $f(\Phi)\psi(\Phi)$ with respect to $\Phi$ on $I$, and, consequently,  $\int_I f(\Phi)\psi(\Phi)\,d\Phi$  also exists, and the integrals are equal  \cite{Bagby}, \cite{Pries}, \cite{Torch1}. 
To make this argument precise we  begin by introducing the partitions used for the approximating Riemann sums. They are based on a partition  $\mathcal Q$ of $\Phi(I)$ defined as follows: given $\eta >0$, by (4),  there is a partition $\mathcal Q=\{\mathcal I_k\}$  of $\Phi(I)$, such that
\begin{equation} \sum_{k} {\rm {osc\, }}(\psi, \mathcal I_k)\, |\mathcal I_k|\le \eta^2 |I|\,.
\end{equation}

We  first separate the indices $ k$ that appear in $\mathcal Q$ into  three classes,    the (good) set $G$, the (bounded) set $B$, and the (undulating) set $U$, according to the following criteria. First, $k\in G$ if  $\psi$ is strictly positive or negative on $\mathcal I_k$. Next, $k\in B$, if $k\notin G$ and $|\psi|\le \eta$ on $\mathcal I_k$. And, finally, $k\in U$, if $k\notin G\cup B$.  Note that for $k\in U$  we have  $ {\rm {osc\, }}(\psi,  \mathcal I_k)\ge \eta$, since $\psi$ changes signs in $\mathcal I_k$ and for at least one point $\zeta_k$ there, $| \psi(\zeta_k)|> \eta$. 

Recall that to each  subinterval  $\mathcal I_k=[\Phi(x_{k,l}),\Phi(x_{k,r})]$
of $\Phi(I)$ corresponds an interval  $I_k=[x_{k,l},x_{k,r}]$, and let $\mathcal P$ denote the partition of $I$ given by $\mathcal P=\{I_k\}$.

Now, since  $f$ is integrable with respect to $\Psi$ on $\Phi( I)$, $f$ is integrable with respect to $\Psi$ on $\mathcal I_k$, and if $k\in G$,  since $\varphi$ and $\psi$ don't change sign, by the Lemma,  $ f(\Phi)\psi(\Phi) $ 
is integrable with respect to $\Phi$ on $I_k$, and 
$\int_{\mathcal I_{k}} f\,d\Psi = \int_{I_k}f(\Phi)\,\psi(\Phi) d\Phi$.
Then, by (4), given $\eta>0$, there is a partition  $\mathcal P^k= \{I_j^k\}$ of $I_k$ such that 
\begin{equation*}
\sum_j {{\rm osc}}\,( f(\Phi)\psi(\Phi),I_j^k )\, |\mathcal I_j^k| \le \eta \,|I_k|\,.
\end{equation*} 
Moreover, since $  \int_{\mathcal I_k} f \,d\Psi\le  U(f(\Phi)\psi(\Phi),\Phi, \mathcal P^k)$, 
we also have
\begin{equation*} 
 U(f(\Phi)\psi(\Phi),\Phi, \mathcal P^k) -\int_{\mathcal I_k} fd\Psi\le\eta\,|I_k|\,.
\end{equation*}

Hence,
\begin{equation}
{{\sum_{k\in G}\sum_j \rm osc}}\,(f(\Phi)\psi(\Phi), I^k_j)\, |\mathcal I^k_j|\le \eta \sum_{k\in G} |I_k|,
\end{equation}
and
\begin{equation}
\sum_{k\in G}\Big| \int_{\mathcal I_{k}} fd\Psi - U(f(\Phi)\psi(\Phi),\Phi, \mathcal P^k)\Big| \le  \eta \sum_{k\in G} |I_k|. 
\end{equation}


Now, for $k\in B\cup U$, let $\mathcal P^k=\{I_k\}$ denote the partition of $I_k$ consisting of the interval $I_k$. Note that, with $M_{\varphi}$ a bound for ${\varphi}$, 
\begin{equation}|\Phi(x_{k,r})-\Phi(x_{k,l})|=|\mathcal I_{k}|\le \int_{[x_{k,l},x_{k,r}]}|\varphi| \le M_{\varphi}\,|I_k|,
\end{equation}
and, with $M_{\psi}$ a bound for ${\psi}$, 
\begin{equation} \Big| \int_{\mathcal I_{k}} f d\Psi \Big|\le M_f \big|\Psi(\Phi(x_{k,r}))-\Psi(\Phi(x_{k,l}))\big|
\le  M_f M_{\psi} M_{\varphi} | I_{k}|.
\end{equation}

Thus, by (16), 
\begin{equation}
{{\rm osc}}\,(f(\Phi)\psi(\Phi), I_k) |\mathcal I_k|\le  2\,  M_f M_{\psi} M_{\varphi}\, |I_k|,
\end{equation}
and   by (17),  for $\xi_k\in I_k$, 
\begin{equation*}\Big| \int_{\mathcal I_{k}} f d\Psi - f(\Phi(\xi_k))\psi(\Phi(\xi_k))\big(\Phi(x_{k,r})-\Phi(x_{k,l})\big)
\Big|\le  2\,  M_f M_{\psi} M_{\varphi}\, |I_k|,
\end{equation*}
and so,   picking  $\xi_k\in I_k$ appropriately,  we get
\begin{equation}\Big| \int_{\mathcal I_{k}} f d\Psi - U( f(\Phi)\psi(\Phi),\Phi\,, \mathcal P^k)\Big|  \le 3 \, 
 M_fM_{\psi} M_{\varphi}|I_k|.
\end{equation}



Now, if $k\in B$,  $M_\psi\le\eta$, and, therefore, by (18),
\begin{equation}
{{\sum_{k\in B} \rm osc}}\,(f(\Phi)\psi(\Phi), I_k)\, |\mathcal I_k|\le  2\,  M_f M_{\varphi}
\,\eta \sum_{k\in B} |I_k|
\end{equation}
and  by (19), 
\begin{equation}
\sum_{k\in B}\Big| \int_{\mathcal I_{k}} fd\Psi - U(f(\Phi)\psi(\Phi),\Phi, \mathcal P^k)\Big|
\le 3\,  M_f M_{\varphi}\, \eta \sum_{k\in B} |I_k|. 
\end{equation}

Finally, since for $k\in U$ we have osc\,($\psi, \mathcal I_k)\ge\eta$, from (13)    it follows that
\[\eta\sum_{k\in U}|\mathcal I_k| \le\sum_{k\in U} {\rm {osc\, }}(\psi,  \mathcal I_k)\, |\mathcal I_k|\le \sum_{k} {\rm {osc\, }}(\psi, \mathcal I_k)\, |\mathcal I_k|\le \eta^2 |I|\,,\]
and, consequently,
\begin{equation} \sum_{k\in U} |\mathcal I_k|\le  \eta\,|I|.
\end{equation} 

Whence, by (18) and (22), the $U$ terms are bounded by 
\begin{equation}\sum_{k\in U} {{\rm osc}}\, (f(\Phi)\psi(\Phi), I_k) |\mathcal I_k|\le  2 M_f M_{\psi} M_{\varphi}\sum_{k\in U} | I_k|\le  2 M_f M_{\psi} M_{\varphi}\,\eta\, |I|,
\end{equation}
and,  by (19) and (22), 
\begin{align}\sum_{k\in U} \Big| \int_{\mathcal I_{k}} f -    U( f(\Phi) &\psi(\Phi),\Phi, \mathcal P^k) \,  \Big|
\nonumber\\
&\le  3\, M_f M_{\psi} M_{\varphi}\,\sum_{k\in U}  |\mathcal I_k| 
  \le 3\, M_f M_{\psi}\,M_{\varphi} \,\eta\, |I|.
\end{align}
 
Consider now the partition $\mathcal P'$ of $I$ that consists of the union of all the intervals in the $\mathcal P^k$, where each $\mathcal P^k$ 
is defined  according as to whether $k \in G, k \in B$, or $k\in U$. Then, by (14), (20), and (23),
\begin{align}\sum_{k\in G}\sum_j {{\rm osc}}\, (f(\Phi)\psi(\Phi),\mathcal I_j^k )\,&|\mathcal I_j^k|  \nonumber\\
+\sum_{k\in B}
{{\rm  osc}} \, (f(\Phi)\psi(\Phi)&, \mathcal I^k)\,|\mathcal I^k|+ \sum_{k\in U} {{\rm osc}}\, (f(\Phi)\psi(\Phi), \mathcal I^k) |\mathcal I^k|\nonumber
\\
\le \eta\sum_{k\in G}|I_k|+& 2\, M_f M_{\varphi}\,\eta \sum_{k\in B}|I_k| + 2\, M_f M_{\psi} M_\varphi\,\eta\, |I|\nonumber
\\
\le 
 \big( 1& + 2\, M_f M_\varphi + 2\, M_f M_\psi M_\varphi\big)\,\eta\,|I|. 
\end{align}

Given $\varepsilon>0$, pick $\eta > 0$ so that
$( 1 + 2\, M_f M_\varphi+ 2\, M_f M_\psi M_{\varphi})\, \eta\,|I|\le \varepsilon$, 
and note that the above expression is $<\varepsilon$, and 
 since $\varepsilon>0$ is arbitrary and $\Phi$ is monotone,  (4) corresponding to $\mathcal P'$  implies that $f(\Phi)\psi(\Phi)$ is Riemann-Stieltjes integrable, and $L(f(\Phi)\psi(\Phi),\Phi)= U(f(\Phi)\psi(\Phi),\Phi) =\int_I f(\Phi)\,\psi(\Phi)\,d\Phi$.

It remains to compute the integral in question. First,  note that 
\begin{equation}U(f(\Phi)\psi(\Phi),\Phi, \mathcal P') =\sum_k U(f(\Phi)\psi(\Phi),\Phi, \mathcal P^k)\,.
\end{equation}
Moreover, since
$\Phi(b)-\Phi(a)=\sum_k \big(\Phi(x_{k,r})-\Phi(x_{k,l})\big)$,
by the linearity of the integral,
taking  orientation into account,   
 it follows that $\int_{\mathcal I}f =\sum_{k} \int_{\mathcal I_{k}} f$, \cite{Robbins}, \cite{Thomson1}. Hence, regrouping according to the sets $G, B$ and $U$, gives 
\begin{equation}
\int_{\mathcal I} fd\Psi =\sum_{k\in G} \int_{\mathcal I_{k}} fd\Psi + \sum_{k\in B} \int_{\mathcal I_{k}} f d\Psi+\sum_{k\in U} \int_{\mathcal I_{k}} f d\Psi, 
\end{equation}
and, from (26) and (27), it follows that
\begin{align*} \Big|\int_{\mathcal I} fd\Psi  -  U( &f(\Phi)\psi(\Phi),\Phi\,, \mathcal P')\Big| \le  
 \sum_{k\in G}\Big| \int_{\mathcal I_k} f d\Psi -  U( f(\Phi)\psi(\Phi),\Phi\,, \mathcal P^k)\Big|
\\
 +  \sum_{k\in B} \Big|& \int_{\mathcal I_k} f d\Psi -  U  ( f(\Phi)\psi(\Phi),\Phi\,, \mathcal P^k)\Big|
\\
&+  \sum_{k\in U}\Big| \int_{\mathcal I_k} f d\Psi - U( f(\Phi)\psi(\Phi),\Phi, \mathcal P^k)\Big|
= s_1 +s_2+ s_3,
\end{align*}
say.   Now, by (15), 
\[  s_1\le  \sum_{k\in G} \Big| \int_{\mathcal I_k}  f d\Psi-  U( f(\Phi)\psi(\Phi),\Phi\,, \mathcal P^k) \Big|\ \le  \eta\, \sum_{k\in G} | I_k|\le \eta\,|I|,\]
and by (21) and (24), $ s_2 +s_3 \le 
\big( 3   M_f M_\varphi + 3 M_f M_\psi M_{\varphi} \big)\, \eta\,|I|,$
which combined imply that 
\[\Big| \int_{\mathcal I} f d\Psi - U(f(\Phi))\psi(\Phi),\Phi, \mathcal P')\Big| \le   ( 1 + 3   M_f M_\varphi + 3 M_f M_\psi M_{\varphi} )\, \eta\,|I|.\]

Given $\varepsilon>0$, pick $\eta > 0$ so that
$ ( 1 + 3   M_f M_\varphi + 3 M_f M_\psi M_{\varphi} )\, \eta\, |I|\le \varepsilon$, 
and note that 
\begin{equation}\Big| \int_{\mathcal I}f d\Psi- U( f(\Phi)\psi(\Phi),\Phi, \mathcal P')\Big| \le  \varepsilon.
\end{equation}
Also, since $U(f(\Phi)\psi(\Phi), \Phi, \mathcal P') -L(f(\Phi)\psi(\Phi), \Phi, \mathcal P')$ is equal to the left-hand side of (25), from (28) it follows that
\begin{equation} \Big| \int_{\mathcal I}f d\Psi- L( f(\Phi)\psi(\Phi),\Phi, \mathcal P')\Big| \le 2 \,\varepsilon\,.
\end{equation}

Furthermore, since by (28),
\begin{align*}\int_I f(\Phi)\psi(\Phi) d\Phi &= U(f(\Phi)\psi(\Phi),\Phi)\\
&\le U(f(\Phi)\psi(\Phi),\Phi, \mathcal P') \le \int_{\mathcal I} f d\Psi+ \varepsilon,
\end{align*}
and by (29),
\begin{align*}\int_{\mathcal I} f d\Psi &\le  L(f(\Phi)\psi(\Phi),\Phi, \mathcal P')+ 2\,\varepsilon
\\
& \le 
 L(f(\Phi)\psi(\Phi),\Phi)+2\, \varepsilon = 
\int_I f(\Phi)\psi(\Phi)\,d\Phi+2\,\varepsilon,
\end{align*}
   we conclude that 
 \[\Big|\int_{\mathcal I} f d\Psi-\int_I f(\Phi)\psi(\Phi)\,d\Phi \Big|\le 2\, \varepsilon,
 \] which, since $ \varepsilon$  is arbitrary, implies that
$\int_I f(\Phi)\psi(\Phi)\,d\Phi =\int_{\mathcal I} f\,d\Psi.$ Hence, (1) holds, and the proof is  finished.
\taf

\section{The Change of Variable Formula.}
The next result corresponds to the case when $\varphi$ is of variable sign, and in this case  the substitution is not required to be invertible.  Then $\Phi(I)$, the
range of $\Phi$, is an interval, but $\Phi(a),  \Phi(b)$ are not necessarily endpoints of
this interval. It is important to keep in mind that the Riemann integral is
oriented, and that the direction in which the interval is traversed determines
the sign of the integral. When $\Psi(x)=x$, the formula is related to the  general formulation by Preiss and Uher \cite{Pries}
of  Kestelman's result pertaining the  change of variable formula for the Riemann integral  \cite{Davies}, \cite{Kestelman}.  
 In fact,  the integral on the right-hand side of (1) can be computed as  a Riemann integral \cite{Lopez1}, \cite{Torch1}, to wit,
\begin{equation*}
 \int_I  f(\Phi)\,\psi(\Phi)\, d\Phi\ = \int_I f(\Phi) \psi(\Phi) \varphi\,.
 \end{equation*}

Specifically, we  have,

\begin{theoremd}  
 Let $\varphi$ be a bounded,  Riemann integrable function defined on an interval $I=[a,b]$, and let $\Phi$ be an indefinite integral of $\varphi$ on $I$. 
Let $\psi$ be a bounded,  Riemann integrable function defined on $\Phi(I)$,
the range of $\Phi$, that does not change sign on $\Phi(I)$,  
and let $\Psi$ be an indefinite integral of $\psi$.
 
Let $f$ be a bounded function  defined on $\Phi(I)$. Then,  $f$ is   Riemann integrable with respect to $\Psi$ on $\Phi(I)$ iff $ f(\Phi)\psi(\Phi) $ is Riemann integrable with respect to $\Phi$ on $I$, and in that case,  
with $\mathcal I=[\Phi(a), \Phi(b)]$, 
\begin{equation} \int_{\mathcal I} f d\Psi= \int_I  f(\Phi)\,\psi(\Phi) d\Phi.
\end{equation}
\end{theoremd}

\begin{proof}
The proof of the necessity follows along the lines of the substitution formula, and we shall be brief. Note that since replacing $\psi$ by $-\psi$ in (30) preserves the identity, it suffices to assume that $\psi$ is positive.  
So,  suppose  that $f$ is integrable with respect to $\Psi$ on $\Phi( I)$, and let  the partition $\mathcal P$ of $I$ be defined as follows: given $\eta >0$, by (4),  there is a partition $\mathcal P=\{I_k\}$  of $I$, such that
\begin{equation} \sum_{k} {\rm {osc\, }}(\varphi,  I_k)\, |I_k|\le \eta^2 |I|\,.
\end{equation}

Separate the indices $ k$ that appear in $\mathcal P$ into  three classes, $G$,  $B$, and   $U$,   according to the following criteria. First, $k\in G$ if  $\varphi$ is strictly positive or negative on $I_k$. Next, $k\in B$, if $k\notin G$ and $|\varphi|\le \eta$ on $I_k$. And, finally, $k\in U$, if $k\notin G\cup B$.  Note that for $k\in U$, since $\varphi$ changes signs in $I_k$ and for at least one point $\xi_k$ there, $| \varphi(\xi_k)|> \eta$, we have  $ {\rm {osc\, }}(\varphi,  I_k)\ge \eta$. 

Recall that each $I_k=[x_{k,l},x_{k,r}]$ in $\mathcal P$   corresponds to the (oriented) subinterval  $\mathcal I_k=[\Phi(x_{k,l}),\Phi(x_{k,r})]$
of $\Phi(I)$. Now, since  $f$ is  integrable with respect to $\Psi$ on $\Phi( I)$, $f$ is integrable with respect to $\Psi$ on $\mathcal I_k$, and if $k\in G$,  since $\varphi$ and $\psi$ don't change sign, by  the Lemma, $ f(\Phi)\psi(\Phi) $ 
is integrable $\Phi$ on $I_k$, and by (7), 
$\int_{\mathcal I_{k}} f\,d\Psi = \int_{I_k}f(\Phi)\,\psi(\Phi) d\Phi$.
Then, by (4), given $\eta>0$, for each $k\in G$, there is a partition  $\mathcal P^k= \{I_j^k\}$ of $I_k$, such that 
$\sum_j {{\rm osc}}\,( f(\Phi)\psi(\Phi),I_j^k )\, |\mathcal I_j^k| \le \eta\, |I_k|,
$ and, therefore, 
\begin{equation}
\sum_{k\in G}\sum_j {{\rm osc}}\,( f(\Phi)\psi(\Phi),I_j^k )\, |\mathcal I_j^k| \le \eta \sum_{k\in G}|I_k|.
\end{equation} 

Now, for $k\in B\cup U$, let $\mathcal P^k=\{I_k\}$ denote the partition of $I_k$ consisting of the interval $I_k$. Then, as in (18) and (19), 
\begin{equation}
{{\rm osc}}\,(f(\Phi)\psi(\Phi), I_k) |\mathcal I_k|\le  2\,  M_f  M_{\psi} M_{\varphi}\, |I_k|\,,
\end{equation}
and   
\begin{equation*}\Big| \int_{\mathcal I_{k}} f d\Psi - U( f(\Phi)\psi(\Phi),\Phi\,, \mathcal P^k)\Big|  \le 3 \, 
 M_f\,M_{\psi} M_{\varphi}\,|I_k|\,.
\end{equation*}

Now, if $k\in B$,  $M_\varphi\le\eta$, and, therefore, by (33), 
\begin{equation}
{{\sum_{k\in B} \rm osc}}\,(f(\Phi)\psi(\Phi), I_k)\, |\mathcal I_k|\le  2\,  M_f  M_{\psi}\,
 \eta \sum_{k\in B} |I_k|
\end{equation}
and  
\begin{equation*}
\sum_{k\in B}\Big| \int_{\mathcal I_{k}} fd\Psi - U(f(\Phi)\psi(\Phi),\Phi, \mathcal P^k)\Big|
\le 3\,  M_f\,M_{\psi}\, \eta \sum_{k\in B} |I_k|\,. 
\end{equation*}

Finally, since for $k\in U$ we have $\sum_{k\in U} |I_k|\le  \eta\,|I|$,
by (33), as in   (23), the $U$ terms are bounded by 
\begin{equation}\sum_{k\in U} {{\rm osc}}\, (f(\Phi)\psi(\Phi), I_k) |\mathcal I_k|\le  2 M_f  M_{\psi} M_{\varphi}\,\eta\, |I|\,,
\end{equation}
and, as in   (24), 
\begin{equation*}\sum_{k\in U} \Big| \int_{\mathcal I_{k}} f -    U( f(\Phi) \psi(\Phi),\Phi, \mathcal P^k) \,  \Big|
  \le 3\, M_f\, M_{\psi}\,M_{\varphi}\,\eta\, |I|.
\end{equation*}
 
Consider now the partition $\mathcal P'$ of $I$ that consists of the union of all the partitions $\mathcal P^k$, where each $\mathcal P^k$ 
is defined  according as to whether $k \in G, k \in B$, or $k\in U$. Then, by (32), (34), and (35), 
\begin{align*}\sum_{k\in G}\sum_j {{\rm osc}}\, (f(\Phi)\psi(\Phi)&,\mathcal I_j^k )\, |\mathcal I_j^k|  +\sum_{k\in B\cup U}
{{\rm  osc}} \, (f(\Phi)\psi(\Phi), \mathcal I^k)\,|\mathcal I^k|\nonumber\\
&\le 
 \big( 1 + 2 M_f M_\psi + 2 M_f M_\psi M_\varphi\big)\,\eta\,|I|\,. 
\end{align*}

Given $\varepsilon>0$, pick $\eta > 0$ so that
$( 1 + 2 M_f M_\psi+ 2 M_f M_\psi M_{\varphi})\,\eta\,|I|\le \varepsilon$, 
and note that the above expression is $<\varepsilon$, and so, 
 since $\varepsilon>0$ is arbitrary,   (4) corresponding to $\mathcal P'$ implies that $f(\Phi)\psi(\Phi)$ is Riemann-Stieltjes integrable on $I$ and $L(f(\Phi)\psi(\Phi),\Phi)= U(f(\Phi)\psi(\Phi),\Phi) =\int_I f(\Phi)\,\psi(\Phi)\,d\Phi$.

Making use of the integral estimates established above, $\int_I f(\Phi)\,\psi(\Phi)\,d\Phi$ can be  evaluated exactly as  in the previous lemma; the computation is left to the reader.

As for the converse, it suffices to prove that,  if $ f(\Phi)\psi(\Phi)$ is integrable with respect to $\Phi$ on $I$,   $f$ is integrable with respect to $\Psi$ on $\Phi(I)$, and  then invoke   the result we  just proved. 
Let the partition $\mathcal P$ of $I$  satisfy (31). Since $\Phi$ is continuous, $\Phi(I)$ is a closed interval of the form $ [ \Phi(x_m),\Phi(x_M) ]$  with (possibly non-unique) $x_m,x_M$ in $I$. If  $x_m$ or $x_M$ is an endpoint  of (not necessarily the same) interval in $\mathcal P$, proceed. Otherwise,
since for an interval $ J=[x_l,x_r]$ and  an interior  point $x$ of $  J$, with $  J_l=[x_l,x]$ and $  J_r=[x,x_r]$, we have 
\begin{equation}{{\rm osc}}\,(f,  J_l)\,|  J_l|+ {{\rm osc}}\, (f,  J_r)\,|  J_r|\le  {\rm osc}\, (f,  J)\,|  J|\,,
\end{equation}
 $\mathcal P$ can be refined so that  the endpoint that was not  originally  included is now an endpoint
of two intervals of the new partition, without increasing the right-hand side
of (31). For simplicity also denote this new partition  $\mathcal P$, note that it contains both $x_m$ and $x_M$ at least once as an endpoint of one of its intervals,   and define  the sets of indices   $G,B$, and $U$,  associated to $\mathcal P$,   as above. 

Now, if $f(\Phi)\psi(\Phi)$ is integrable with respect to $\Phi$ on $I$, $ f(\Phi)\psi(\Phi)$ is integrable with respect to $\Phi$ on $I_k$, and, if $k\in G$, since $\varphi$ is of constant sign,  by the lemma, $ f $ 
is integrable with respect to $\Psi$ on $\mathcal I_k$ and
$\int_{I_k}f(\Phi)\,\psi(\Phi)\,d\Phi =\int_{\mathcal I_{k}} f d\Psi$. 
Then, by (4), given $\eta>0$, there is a partition  $\mathcal Q^k= \{{\mathcal I}_j^k\}$ of $\mathcal I_k$, where ${\mathcal I}_j^k=[\Phi(x_{j,l}^k),\Phi(x_{j,r}^k)]$,            such that 
\begin{equation} \sum_j {{\rm osc}}\, (f, \mathcal I_j^k)\,\big|\Psi(\Phi(x_{j,r}^k))-\Psi(\Phi(x_{j,l}^k))\big|\le \eta \,|I_k|\,.
\end{equation}

As for $k\in B\cup U$,  as in (34), 
\[ {{\rm osc }}\, (f, \mathcal I_k)\,\big|\Psi(\Phi(x_{k,r}))-\Psi(\Phi(x_{k,l}))\big| \le 2 M_f M_{\psi}M_\varphi\,|I_k|\,.
\]
Next, if $k\in B$,   $M_\varphi\le\eta$, and, therefore, 
\begin{equation} \sum_{k\in B} {{\rm osc}}\, (f, \mathcal I_k)\,\big|\Psi(\Phi(x_{k,r}))-\Psi(\Phi(x_{k,l}))\big|\le  2 M_f M_\psi \,\eta \sum_{k\in B}|I_k|\,.
\end{equation}
Finally, for $k\in U$, as in (35), 
\begin{align}\sum_{k\in U} {{\rm osc}}\,( f,\mathcal I_k)& \big|\Psi(\Phi(x_{k,r}))-\Psi(\Phi(x_{k,l}))\big|\nonumber
\\&\le 2\, 
 M_f\,M_\psi\, M_{\varphi}\sum_{k\in U} |I_k|\le  2\, M_f\,M_\psi\,M_{\varphi}\,\eta\, |I|.
\end{align}
 
Let $\mathcal Q'$ denote the collection   of subintervals of $\Phi(I )$  defined by 
\[\mathcal Q'= \big( \bigcup_{k\in G} \bigcup_j  \{\mathcal I_j^k\} \big) \cup \big( \bigcup_{k\in B\cup U}\{\mathcal I_k\}\big).\] 
Note that the union of the intervals in $\mathcal Q'$ is $\Phi(I)$  and that, by (37), (38), and (39),  
\begin{align} \sum_{k\in G}\sum_j {{\rm osc}}\, (f,\mathcal I_j^k )\, &\big|\Psi(\Phi(x_{k,r}^j))-\Psi(\Phi(x_{k,l}^j))\big|\nonumber
\\+\sum_{k\in B\cup U}&{{\rm osc}}\,( f,\mathcal I_k) \big|\Psi(\Phi(x_{k,r}))-\Psi(\Phi(x_{k,l}))\big|
\nonumber\\
&\le  \big( 1 + 2 M_f M_\psi + 2 M_fM_\psi M_\varphi\big)\,\eta\,|I|. 
\end{align}

Consider now the finite set  $\Phi(x_m)=y_1< y_2<  \cdots  <\Phi(x_M)=y_h$,  of   the endpoints of the intervals in $\mathcal Q'$  arranged in an increasing fashion, without repetition. Suppose that the interval $\mathcal J$ in $\mathcal Q'$ contains the points $y_{k_1}, \ldots, y_{k_n}$, say,  as endpoints or interior points. If they are endpoints, disregard them, otherwise, as in (36),  incorporate each as an  endpoint of two intervals in a refined  $\mathcal Q'$ without  increasing the right-hand side of (40). Clearly $\mathcal Q'$ thus refined contains a partition $\mathcal Q'' =\{\mathcal J_k\}$ of $\Phi(I)$, which, by    (40), satisfies,
\[ \sum_k {\rm osc}\,(f,\mathcal J_k)\,|\mathcal J_k|\le \big( 1 + 2 M_f M_\psi + 2 M_f M_\psi M_\varphi\big)\,\eta\,|I|\,. 
\]

Given $\varepsilon>0$, pick $\eta>0$ such that  $\big( 1+ 2 M_f M_\psi +2 M_f M_\psi M_{\varphi}\big)\,\eta\,|I|\le \varepsilon$.
 Then  the sum in (4) corresponding to $\mathcal Q''$ does not exceed an arbitrary $\varepsilon>0$, and, therefore, $f$ is integrable with respect to $\Psi$ on $\Phi(I)$, and the conclusion  follows from the first part of the proof.
\taf \end{proof}

\section{Coda.}
We close this note with a  caveat: not always the most general result is the most useful. By strengthening some assumptions and  weakening others in the change of variable formula, it is possible to obtain a  substitution formula that does not follow from this result \cite{deO}.  

Assume  that the function $\Phi$ is continuous,  increasing on $I=[a,b]$, and differentiable on $(a,b)$  with derivative $\varphi\ge 0$;  then $\Phi$ is  uniformly continuous on $I$, and  maps $I$ onto $\mathcal I=[\Phi(a),\Phi( b)]$. Assume  that $\Psi$ is continuous,  increasing on $\mathcal I$, and differentiable on $(\Phi(a),\Phi(b))$  with derivative $\psi\ge 0$.  We will also assume that $f$ is  Riemann integrable, rather than bounded, on $ \mathcal I$. On the other hand, we will not assume  that $\varphi,\psi$ are bounded.
Then,  if   $f(\Phi)\psi(\Phi)$ is integrable with respect to $\Phi$ on $I$, the change of variable  formula holds.



To see this, consider a partition $\mathcal P=\{I_k\}$, $I_k=[x_{k,l},x_{k,r}]$, of $I$, and  the corresponding partition $\mathcal Q=\{\mathcal I_k\}$ 
of $\mathcal I$,  consisting of 
$\mathcal I_k=[y_{k,l},y_{k,r}]$, where  $y_{k,l}=\Phi(x_{k,l})$ and  $y_{n,r}=\Phi(x_{k,r})$. 
By the mean value theorem  there exist $\zeta_k\in \mathcal I_k$ such that 
\begin{equation*}
\Psi(y_{k,r})-\Psi(y_{k,l})= \psi(\zeta_k)\,\big(y_{k,r}-y_{k,l}\big)\,,\quad {{\rm all }}\ k\,,
\end{equation*}
and with $\xi_k\in I_k$ such that $\zeta_k=\Phi(\xi_k)$, all $k$, it follows that 
\[ \sum_k f(\zeta_k)\,\big( \Psi(y_{k,r})-\Psi(y_{k,l}) \big)= \sum_k   f(\Phi(\xi_k ))\,\psi(\Phi(\xi_k))\,\big(\Phi(x_{k,r})-\Phi(x_{k,l})\big), 
\]
where the left-hand side is a Riemann sum of $f$ with respect to $\Psi$ on $\mathcal I$, and the right-hand side a Riemann sum of $f(\Phi)\psi(\Phi)$ with respect to $\Phi$ on $ I$.
Since by the  uniform continuity of $\Phi$ it follows that $\max_k |I_k|\to 0$  
implies $\max_k|\mathcal I_k|\to 0$,  by the integrability assumptions, for appropriate partitions $\mathcal P$ the left-handside above tends to
$\int_{\mathcal I} f\,d\Psi$, and the right-hand side to $\int_{I}f(\Phi)\psi(\Phi)\,d\Phi$. Hence the change of variable formula holds. 

This observation applies in the  following setting. On $I=\mathcal I=[\,0,1]$, with $0<\varepsilon, \eta<1$,  let $ \Phi(x)= x^{1-\varepsilon}$,  $\varphi(x)= (1-\varepsilon)\, x^{-\varepsilon}$ for $x\in (0,1]$, and $\Psi(y)=y^{1-\eta}$,  $\psi(y)=(1-\eta)\,y^{-\eta}$ for $y\in (0, 1]$; $\varphi$ and $\psi$ are unbounded.
Then, for an integrable function $f$ on $\mathcal I$, provided that $f(\Phi)\,\psi(\Phi)$ is integrable with repect to $\Phi$ on $I$, the change of variable formula holds. For $f$ we may take a  continuous function of order $x^\beta$ near the origin,  where $\beta\ge \varepsilon/(1-\varepsilon)+\eta$.

\end{document}